\documentclass[11pt,reqno]{amsart}

\usepackage{amsmath, amsthm, amsopn, amssymb}
\usepackage{enumerate}
\usepackage{hyperref}

\newcommand{\cA}{\mathcal{A}}

\newcommand{\HH}{\mathcal{H}}
\newcommand{\II}{\mathcal{I}}
\newcommand{\ex}{\mathrm{ex}}
\newcommand{\cP}{\mathcal{P}}
\newcommand{\cS}{\mathcal{S}}
\newcommand{\SF}{\mathrm{SF}}
\newcommand{\ZZ}{\mathbb{Z}}

\newcommand{\np}{n'}

\renewcommand{\le}{\leqslant}
\renewcommand{\ge}{\geqslant}

\newcommand{\eps}{\varepsilon}

\newtheorem{thm}{Theorem}
\newtheorem{lemma}[thm]{Lemma}
\newtheorem{prop}[thm]{Proposition}

\theoremstyle{definition}
\newtheorem*{dfn}{Definition}
\newtheorem*{alg}{The Algorithm}
\newtheorem*{claim}{Claim}

\title{Counting independent sets in graphs}
\author{Wojciech Samotij}
\address{School of Mathematical Sciences, Tel Aviv University, Tel Aviv 69978, Israel; and Trinity College, Cambridge CB2 1TQ, UK}
\email{samotij@post.tau.ac.il}
\thanks{Research supported in part by a grant from the Israel Science Foundation}
\date{\today}

\begin{document}

\maketitle

\begin{abstract}
  In this short survey article, we present an elementary, yet quite powerful, method of enumerating independent sets in graphs. This method was first employed more than three decades ago by Kleitman and Winston and has subsequently been used numerous times by many researchers in various contexts. Our presentation of the method is illustrated with several applications of it to `real-life' combinatorial problems. In particular, we derive bounds on the number of independent sets in regular graphs, sum-free subsets of $\{1, \ldots, n\}$, and $C_4$-free graphs and give a short proof of an analogue of Roth's theorem on $3$-term arithmetic progressions in sparse random sets of integers which was originally formulated and proved by Kohayakawa, \L uczak, and R\"odl.
\end{abstract}

\section{Introduction}
\label{sec:introduction}

Many well-studied problems in combinatorics concern characterising discrete structures that~satisfy certain `local' constraints. For example, the~celebrated theorem of~Szemer\'edi~\cite{Sz75} gives an~upper bound on the~maximum size of~a~subset of~the~first $n$ integers which does not contain an~arithmetic progression of~a~fixed length~$k$. To~give another example, the~archetypal problem studied in~extremal graph theory, dating back to~the~work of Mantel~\cite{Ma07} and Tur\'an~\cite{Tu41}, is that of~characterising graphs which do not contain a~fixed graph $H$ as~a~subgraph.

Problems of this type fall into the following general framework. We are given a~finite set~$V$ and a~collection~$\HH$ of~subsets of~$V$. What can be said about sets~$I \subseteq V$ that do not contain any~member of~$\HH$? Such a~collection~$\HH$ is often called a~\emph{hypergraph} with vertex set~$V$, members of~$\HH$ are termed \emph{edges}, and any set~$I \subseteq V$ that contains no edge is called an~\emph{independent set}. In view of this, one might say that a~large part of~combinatorics is concerned with studying independent sets in various hypergraphs. For instance, in the first example from the previous paragraph, $V$ is the set $\{1, \ldots, n\}$ and $\HH$ is the collection of all $k$-term arithmetic progressions contained in $V$; stated in~this language, Szemer\'edi's theorem says that for every positive constant $\delta$, every independent set in $\HH$ has fewer than $\delta n$ elements, provided that $n$ is sufficiently large. In the second example, $V$ is the edge set of a complete graph on a given set of $n$ vertices and $\HH$ is the family of all $\binom{n}{|V(H)|}$~sets of $|E(H)|$ edges that form a copy of $H$ in the complete graph; in this notation, if $H$ is a clique with $k+1$ vertices, then Tur\'an's theorem says that the largest independent sets in $\HH$ are precisely the edge sets of the complete balanced $k$-partite subgraphs of the complete graph with edge set $V$ and the well-known theorem of Kolaitis, Pr\"omel, and Rothschild~\cite{KoPrRo87} states that almost all independent sets of $\HH$ are $k$-partite, that is, the number $i^*(\HH)$ of independent sets in $\HH$ that are not the edge sets of $k$-partite subgraphs of the complete graph with edge set $V$ satisfies $i^*(\HH) / i(\HH) \to 0$ as $n \to \infty$.

For a~hypergraph~$\HH$, let~$\II(\HH)$ denote the~family of all independent sets in~$\HH$, let $i(\HH) = |\II(\HH)|$, and let $\alpha(\HH)$ be the largest cardinality of an element of~$\II(\HH)$, usually called the~\emph{independence number} of~$\HH$. There are two natural problems that one usually poses about a specific hypergraph~$\HH$:
\begin{enumerate}[(i)]
\item\label{item:q1}
  Determine $\alpha(\HH)$ and describe all $I \in \II(\HH)$ with $\alpha(\HH)$ elements.
\item\label{item:q2}
  Estimate $i(\HH)$ and describe a `typical' member of $\II(\HH)$.
\end{enumerate}
Let us remark here that providing a precise characterisation of a typical element of $\II(\HH)$ usually yields a~precise estimate for~$i(\HH)$.

An~apparent connection between problems~(\ref{item:q1}) and~(\ref{item:q2}) may be easily observed in the following two inequalities, which are trivial consequences of the above definitions and the fact that the family~$\II(\HH)$ is closed under taking subsets:
\begin{equation}
  \label{eq:alpha-i}
  2^{\alpha(\HH)} \le i(\HH) \le \sum_{m = 0}^{\alpha(\HH)} \binom{|V(\HH)|}{m}.
\end{equation}
Note that, unless $\alpha(\HH)$ is very close to~$|V(\HH)|$, the~lower and upper bounds on~$i(\HH)$ given in~\eqref{eq:alpha-i} are quite far apart. Since for many interesting hypergraphs $\HH$ this naive lower bound is actually fairly close to being best possible, the~efforts of~many researchers have been focused on improving the upper bound.

In this short survey article, we present an elementary, yet very powerful, method for~proving stronger upper bounds in the~case when all edges of~$\HH$ have size two, that is, when $\HH$ is a~graph. This method was first described more than three decades ago by~Kleitman and Winston, who used it to obtain upper bounds on the number of lattices\footnote{A lattice is a partially ordered set in which every two elements have a supremum and an infimum.}~\cite{KlWi80} and graphs without cycles of~length four~\cite{KlWi82}. Variations of this method were subsequently rediscovered by several researchers, most notably by Sapozhenko, in the context of enumerating independent sets in regular graphs~\cite{Al91, Sa01} and sum-free sets in abelian groups~\cite{Al91, LeLuSc01, Sa02}. We shall illustrate our presentation of this method with several applications of it to `real-life' combinatorial problems. We would like to stress here that none of the results or proof techniques presented here are new, but we hope that there is some value in seeing them next to one another.

\section{The Kleitman--Winston algorithm}
\label{sec:KW}

Suppose that we are given an arbitrary graph $G$ with $n$ vertices. Our goal is to give an upper bound on $i(G)$, the number of independent sets in $G$. The idea of Kleitman and Winston was to devise an algorithm that, given a particular independent set $I \in \II(G)$, would encode $I$ in an invertible way. Crucially, the encoding should be performed in a way which makes it easy to estimate the total number of outputs of the algorithm. Since for every invertible encoding, the total number of outputs is precisely $i(G)$, in this way one could derive an upper bound on this quantity.

The crucial idea of Kleitman and Winston was to consider the vertices of $G$ ordered according to their degrees and encode each independent set $I$ as a sequence of positions of the elements of $I$ in that ordering. We make this precise below.

\begin{dfn}
  Let $G$ be a graph and fix an arbitrary total order on $V(G)$. For every $A \subseteq V(G)$, the \emph{max-degree ordering} of $A$ is the ordering $(v_1, \ldots, v_{|A|})$ of all elements of $A$, where for each $j \in \{1, \ldots, |A|\}$, $v_j$ is the maximum-degree vertex in the subgraph of $G$ induced by $A \setminus \{v_1, \ldots, v_{j-1}\}$; ties are broken by giving preference to vertices that come earlier in the fixed total order on $V(G)$.
\end{dfn}

\begin{alg}
  Suppose that a graph $G$, an $I \in \II(G)$, and an integer $q \le |I|$ are given. Set~$A = V(G)$ and $S = \emptyset$. For $s = 1, \ldots, q$, do the following:
  \begin{enumerate}[(a)]
  \item
    Let $(v_1, \ldots, v_{|A|})$ be the max-degree ordering of $A$.
  \item
    Let $j_s$ be the minimal index $j$ such that $v_j \in I$.
  \item\label{item:v}
    Move $v_{j_s}$ from $A$ to $S$.
  \item\label{item:before-v}
    Delete $v_1, \ldots, v_{j_s-1}$ from $A$.
  \item\label{item:Nv}
    Delete $N_G(v_{j_s}) \cap A$ from $A$.
  \end{enumerate}
  Output $(j_1, \ldots, j_q)$ and $A \cap I$.
\end{alg}

For each output sequence $(j_1, \ldots, j_q)$ and every $s \in \{1, \ldots, q\}$, denote by $A(j_1, \ldots, j_s)$ and $S(j_1, \ldots, j_s)$ the sets $A$ and $S$ at the end of the $s$th iteration of the algorithm (run on some input~$I$ that produces this particular sequence $(j_1, \ldots, j_q)$), respectively. Observe that these definitions do not depend on the choice of $I$ as the sequence $(j_1, \ldots, j_q)$ uniquely determines how the sets $S$ and $A$ evolve throughout the algorithm. More precisely, if running the algorithm on two inputs $I, I' \in \II(G)$ produces the same sequence $(j_1, \ldots, j_q)$, then both these executions will also yield the same sets $S$ and $A$. Indeed, all the modifications of the sets $S$ and $A$ in the~$s$th iteration of the algorithm depend solely on $j_s$. 

Note crucially that $S(j_1, \ldots, j_s) \subseteq I$ and $I \setminus S(j_1, \ldots, j_s) \subseteq A(j_1, \ldots, j_s)$ for every $s$. Indeed, by the minimality of $j_s$ and the assumption that $I$ is independent, the only vertices of $I$ that are deleted from $A$ are moved to~$S$. It follows that one may recover the set $I$ from the output of the algorithm, as $I = S(j_1, \ldots, j_q) \cup (A(j_1, \ldots, j_q) \cap I)$. We also note for future reference that the sequence $(j_1, \ldots, j_q)$ can be recovered from the set $S(j_1, \ldots, j_q)$. Indeed, if running the algorithm on some input $I \in \II(G)$ produces a sequence $(j_1, \ldots, j_q)$ and $S = S(j_1, \ldots, j_q)$, then the same sequence will be produced by running the algorithm with $I$ replaced by~$S$. Finally, let us observe that $j_1 + \ldots + j_q \le |V(G)| - |A(j_1, \ldots, j_q)|$, as in steps (\ref{item:v}) and (\ref{item:before-v}) of the $s$th iteration of the main loop, we removed from $A$ some $j_s$ vertices.

Let $i(G,m)$ be the number of independent sets in $G$ that have precisely $m$ elements. The above observations readily imply that for every $m$ and $q$ with $m \ge q$,
\begin{equation}
  \label{eq:iG-q-m}
  i(G,m) \le \sum_{(j_s)} i\big(G[A(j_1, \ldots, j_q)], m-q\big) \le \sum_{(j_s)} \binom{|A(j_1, \ldots, j_q)|}{m-q},
\end{equation}
where the above sums range over all output sequences $(j_1, \ldots, j_q)$. In particular, letting $n = |V(G)|$,
\begin{equation}
  \label{eq:iG-q}
  i(G) \le \sum_{m=0}^{q-1} \binom{n}{m} + \sum_{(j_s)} i\big(G[A(j_1, \ldots, j_q)]\big) \le \sum_{m=0}^{q-1} \binom{n}{m} + \sum_{(j_s)} 2^{|A(j_1, \ldots, j_q)|}.
\end{equation}

In view of~\eqref{eq:iG-q-m} and~\eqref{eq:iG-q}, it is in our interest to make the set $A(j_1, \ldots, j_q)$ as small as possible, uniformly for all values of $(j_1, \ldots, j_q)$. This is why we consider the vertices of $A$ listed according to the max-degree ordering. (An attentive reader might have already noticed that this particular ordering maximises $\deg_G(v_{j_s}, A)$ in each iteration of the algorithm.) Suppose that we are at the $s$th iteration of the main loop of the algorithm and let $A' = A \setminus \{v_1, \ldots, v_{j_s-1}\}$, where $A$ is as at the start of this iteration, that is, $A = A(j_1, \ldots, j_{s-1})$. By the definition of the max-degree ordering,
\[
|N_G(v_{j_s}) \cap A'| = \max_{v \in A'} \deg_G(v,A') \ge \frac{2e_G(A')}{|A'|}.
\]
In particular, if $e_G(A') = \beta \binom{|A'|}{2}$, then the right-hand side of the above inequality is $\beta (|A'| - 1)$. Consequently, the number of vertices that are removed from $A$ during the $s$th iteration of the main loop of the algorithm is at least $j_s  + \beta (|A'|-1)$, which is at least $\beta |A|$, as $|A'| - 1 = |A| - j_s$ and~$\beta \le 1$. In other words, as long as the density of the subgraph induced by the set $A$ exceeds some~$\beta$, each iteration of the main loop of the algorithm shrinks $A$ by a factor of at most $1 - \beta$.

The following two lemmas, which are both implicit in the work of Kleitman and Winston, summarise the above discussion. The first lemma gives a simple bound on the number of independent sets of a given size in a graph which satisfies a certain local density condition. The exact statement of this lemma is taken from~\cite{KoLeRoSa14}. The second lemma characterises the family of all independent sets in such a locally dense graph. The statement of this lemma is inspired by the statement of the~main result of~\cite{BaMoSa14}.

\begin{lemma}
  \label{lemma:KW-basic}
  Let $G$ be a graph on $n$ vertices and assume that an integer $q$ and reals $R$ and $\beta \in [0,1]$ satisfy
  \begin{equation}
    \label{eq:beta-q-R}
    R \ge e^{-\beta q} n.
  \end{equation}
  Suppose that the number of edges induced in $G$ by every set $U \subseteq V(G)$ with $|U| \ge R$ satisfies
  \begin{equation}
    \label{eq:eGU}
    e_G(U) \ge \beta \binom{|U|}{2}.
  \end{equation}
  Then, for every integer $m$ with $m \ge q$,
  \begin{equation}
    \label{eq:iG-m-bound}
    i(G,m) \le \binom{n}{q} \binom{R}{m-q}.
  \end{equation}
\end{lemma}

\begin{proof}
  Since there are exactly $\binom{n}{q}$ sequences $(j_1, \ldots, j_q)$ satisfying $j_1 + \ldots + j_q \le n$ and $j_s \ge 1$ for each $s$, the sum in the right-hand side of~\eqref{eq:iG-q-m} has at most $\binom{n}{q}$ terms. Therefore, it suffices to show that for each sequence $(j_1, \ldots, j_q)$ that is outputted by the algorithm, the set $A(j_1, \ldots, j_q)$ has at most $R$ elements. If this were not the case, then there would be some sequence $(j_1, \ldots, j_q)$ such that for each $s \in \{1, \ldots, q\}$, the set $A \setminus \{v_1, \ldots, v_{j_s-1}\}$ in the $s$th iteration of the main loop of the algorithm (run on some input that results in this particular sequence) would have more than $R$ elements and therefore induce in $G$ a subgraph with edge density at least $\beta$. It follows from our discussion that each of the $q$ iterations would shrink the set~$A$ by a~factor of~at~most $1-\beta$. Since $|A| = |V(G)| = n$ at the~start of the~algorithm, then, by~\eqref{eq:beta-q-R},
  \[
  |A(j_1, \ldots, j_q)| \le (1-\beta)^q n \le e^{-\beta q}n \le R,
  \]
  a contradiction.
\end{proof}

\begin{lemma}
  \label{lemma:KW-containers}
  Let $G$ be a graph on $n$ vertices and assume that an integer $q$ and reals $R$ and $D$ satisfy
  \begin{equation}
    \label{eq:D-q-R}
    R + q D \ge n.
  \end{equation}
  Suppose that the number of edges induced in $G$ by every set $U \subseteq V(G)$ with $|U| \ge R$ satisfies
  \begin{equation}
    \label{eq:eGU-D}
    2e_G(U) \ge D|U|.
  \end{equation}
  Then there exists a collection $\cS$ of $q$-element subsets of $V(G)$ and two mappings $g \colon \II(G) \to \cS$ and $f\colon \cS \to \cP(V(G))$ such that $|f(S)| \le R$ for each $S \in \cS$ and $g(I) \subseteq I \subseteq f(g(I)) \cup g(I)$ for every $I \in \II(G)$ with at least $q$ elements.
\end{lemma}

\begin{proof}
  We define the mappings $f$ and $g$ and the family $\cS$ as follows. We simply run the algorithm with input $I$ for each $I \in \II(G)$ with at least $q$ elements and let $g(I)$ and $f(g(I))$ be the final sets $S$ and $A$, respectively. Moreover, we let $\cS$ be the family of all such $S$, that is, the set of values taken by $g$. The discussion in the paragraph following the description of the algorithm should convince us that this is a valid definition of $f$, that $g(I) \subseteq I \subseteq f(g(I)) \cup g(I)$ for each $I$ as above, and that $\cS$ consists solely of $q$-element subsets of $V(G)$. It suffices to check that $|f(g(I))| \le R$ for each such $I$. If this were not the case, then there would be some sequence $(j_1, \ldots, j_q)$ such that for each $s \in \{1, \ldots, q\}$, the set $A \setminus \{v_1, \ldots, v_{j_s-1}\}$ in the $s$th iteration of the main loop of the algorithm (run on an input $I$ that generates this sequence) would have more than $R$ elements and therefore induce in $G$ a subgraph with average degree at least $D$. But then, each of the $q$ iterations would remove from $A$ at least $D+1$ vertices. Since $|A| = |V(G)| = n$ at the start of the algorithm, then by~\eqref{eq:D-q-R},
  \[
  |A(j_1, \ldots, j_q)| \le n - D q \le R,  
  \]
  a contradiction.
\end{proof}

Before we close this section, let us make several final remarks. First, the conclusion of Lemma~\ref{lemma:KW-containers} is stronger than the conclusion of Lemma~\ref{lemma:KW-basic}. This is simply because the existence of $f$ and $g$ as in the statement of the second lemma imply the bound on $i(G,m)$ asserted by the first lemma. Moreover, it should be clear from the proofs that the assumptions of the two lemmas are `interchangeable' in the following sense. If a graph $G$ satisfies the assumptions of Lemma~\ref{lemma:KW-basic} with some $q$, $R$, and $\beta$, then the conclusion of Lemma~\ref{lemma:KW-containers} holds for $G$ with the same $q$ and $R$; and vice-versa, if a graph $G$ satisfies the assumptions of Lemma~\ref{lemma:KW-containers} with some $q$, $R$, and $D$, then the conclusion of Lemma~\ref{lemma:KW-basic} holds for~$G$ with the same $q$ and $R$. (The latter statement is redundant because, as we have already noted above, the conclusion of Lemma~\ref{lemma:KW-containers} is stronger than the conclusion of Lemma~\ref{lemma:KW-basic}.)

\section{Applications}
\label{sec:applications}

\subsection{Independent sets in regular graphs}
\label{sec:indep-sets-reg-graphs}

During a number theory conference at Banff in~1988, Granville conjectured (see~\cite{Al91}) that an $n$-vertex $d$-regular graph can have no more than $2^{(1+o(1))\frac{n}{2}}$ independent sets, where $o(1)$ is some function that tends to $0$ as $d \to \infty$. A few years later, this was shown to be true by Alon~\cite{Al91}, who proved that in fact
\[
i(G) \le 2^{(1+O(d^{-0.1}))\frac{n}{2}}
\]
for every $n$-vertex $d$-regular graph $G$. As our first application of Lemma~\ref{lemma:KW-basic}, we derive a somewhat stronger estimate, which was obtained several years later by Sapozhenko~\cite{Sa01}, using arguments very similar to those presented in Section~\ref{sec:KW}.

\begin{thm}[{\cite{Sa01}}]
  \label{thm:Sapozhenko}
  There is an absolute constant $C$ such that every $n$-vertex $d$-regular graph $G$ satisfies
  \[
  i(G) \le 2^{\left(1 + C \sqrt{\frac{\log d}{d}}\right)\frac{n}{2}}.
  \]
\end{thm}

Alon~\cite{Al91} speculated that when $n$ is divisible by $2d$, then the disjoint union of $\frac{n}{2d}$ complete bipartite graphs $K_{d,d}$ has the maximum number of independent sets among all $d$-regular graphs with $n$~vertices. A slightly stronger statement (Theorem~\ref{thm:Kahn-Zhao} below) was later conjectured by Kahn~\cite{Ka01}, who proved it under the additional assumption that $G$ is bipartite, using a beautiful entropy argument. This assumption was recently shown to be unnecessary by Zhao~\cite{Zh10}, who gave a short and elegant argument showing that for every $n$-vertex $d$-regular graph $G$, there exists a $2n$-vertex $d$-regular bipartite graph $G'$ such that $i(G) \le i(G')^{1/2}$. The results of Kahn and Zhao yield the following.

\begin{thm}[{\cite{Ka01,Zh10}}]
  \label{thm:Kahn-Zhao}
  For every $n$-vertex $d$-regular graph $G$,
  \[
  i(G) \le i(K_{d,d})^{\frac{n}{2d}} = \left(2^{d+1}-1\right)^{\frac{n}{2d}}.
  \]
\end{thm}

We now derive Theorem~\ref{thm:Sapozhenko} from Lemma~\ref{lemma:KW-basic}.

\begin{proof}[Proof of Theorem~\ref{thm:Sapozhenko}]
  Let $G$ be an $n$-vertex $d$-regular graph. We shall in fact estimate $i(G,m)$ for each $m$ and deduce the claimed bound on $i(G)$ by summing over all $m$. Since $i(G) \le 2^n$ and $C$ is an arbitrary constant, we may assume that $d$ is sufficiently large (and therefore $n$ is sufficiently large). We consider two cases. First, if $m \le n/10$, then we simply note that
  \begin{equation}
    \label{eq:ind-set-reg-1}
    i(G,m) \le \binom{n}{\frac{n}{10}} \le (10e)^{\frac{n}{10}} \le 2^{0.48n},
  \end{equation}
  where we used the well-known inequality $\binom{a}{b} \le (ea/b)^b$ valid for all $a$ and $b$.

  In the complementary case, $m > n/10$, we shall apply Lemma~\ref{lemma:KW-basic}. To this end, let $B \subseteq V(G)$ and note that
  \begin{equation}
    \label{eq:degree-sums}
    d|B| = \sum_{v \in B} \deg_G(v) = 2e(B) + e(B,B^c) \le 2e(B) + \sum_{v \in B^c} \deg_G(v) = 2e(B) + d(n-|B|).
  \end{equation}
  Fix an arbitrary $\beta$, let $R = \frac{n}{2} + \frac{\beta n^2}{2d}$, and observe that if $|B| \ge R$, then~\eqref{eq:degree-sums} yields
  \begin{equation}
    \label{eq:eB-bound}
    e(B) \ge \frac{d}{2}(2|B| - n) \ge \frac{d}{2}(2R - n) \ge \frac{\beta n^2}{2} \ge \beta\binom{|B|}{2}.
  \end{equation}
  Assume that $\beta > 10/n$ and let $q = \lceil 1/\beta \rceil$. By Lemma~\ref{lemma:KW-basic}, since
  \[
  e^{-\beta q} n \le e^{-1}n \le R,
  \]
  then for every $m$ with $m \ge \lceil n/10 \rceil \ge q$,
  \begin{equation}
    \label{eq:ind-set-reg-2}
    i(G,m) \le \binom{n}{q} \binom{\frac{n}{2} + \frac{\beta n^2}{2d}}{m-q} \le \left(\frac{en}{q}\right)^q \binom{\frac{n}{2} + \frac{\beta n^2}{2d}}{m-q} \le  (e\beta n)^{\lceil 1/\beta \rceil} \cdot \binom{\frac{n}{2} + \frac{\beta n^2}{2d}}{m-q}.
  \end{equation}
  Summing~\eqref{eq:ind-set-reg-1} and~\eqref{eq:ind-set-reg-2} over all $m$ yields
  \[
  i(G) \le 2^{0.49n} + 2^{\frac{n}{2} + \frac{\beta n^2}{2d} + \lceil 1/\beta \rceil \log_2(e\beta n)}
  \]
  We obtain the claimed bound by letting $\beta = \frac{\sqrt{d \log d}}{n}$; we note that $\sqrt{d \log d} > 10$ as we assumed that $d$ is large.
\end{proof}

We ought to indicate here that one may significantly improve the upper bound given by Theorem~\ref{thm:Sapozhenko} by a somewhat more careful analysis of the execution of the Kleitman--Winston algorithm than the one given in the proof of Lemma~\ref{lemma:KW-basic}. The main reason why one should expect such an~improvement to be possible is the crudeness of the second inequality in~\eqref{eq:eB-bound} in the case when $|B| - n/2$ is much larger than $R - n/2$. The proof of Lemma~\ref{lemma:KW-basic} uses~\eqref{eq:eB-bound} to show that in each step of the~algorithm, the set $A$ loses at least $\beta |A|$ elements whereas in reality $A$ will lose many more elements as long as $|A|$ is not very close to $n/2 + \beta n^2 / (2d)$. By considering the `evolution' of $|A|$ partitioned into `dyadic' intervals $\big(n/2 + n/2^{i+1}, n/2 + n/2^i\big]$, where $1 \le i \le \log_2 d - \log_2 \log_2 d$, one may prove that there is an absolute constant $C$ such that every $n$-vertex $d$-regular graph $G$ satisfies
\[
i(G) \le 2^{\left(1 + C\frac{(\log d)^2}{d}\right)\frac{n}{2}}.
\]
One rigorous way of tracking this `evolution' of $|A|$ is to repeatedly invoke Lemma~\ref{lemma:KW-containers} with $R_i = n/2 + n/2^{i+1}$ and $D_i = d/2^i$ for $i = 1, \ldots, \log_2 d - \log_2 \log_2 d$. We leave filling in the details as an~exercise for the reader.

\subsection{Sum-free sets}

The conjecture of Granville mentioned in the previous section was motivated by a problem posed by Cameron and Erd\H{o}s at the same number theory conference. A set $A$ of elements of an abelian group is called \emph{sum-free} if there are no $x,y,z \in A$ satisfying $x + y = z$. Let $[n]$ denote the set $\{1, \ldots, n\} \subseteq \ZZ$. Cameron and Erd\H{o}s raised the question of determining the~number $\SF([n])$ of sum-free sets contained in the set $[n]$. They noted that any set containing either only odd integers or only integers greater than $n/2$ is sum-free, and that it is unlikely that there is another large collection of sum-free sets that are not essentially of one of the above two types. In view of this, they conjectured that $\SF([n]) = O(2^{n/2})$. Soon afterwards, Alon~\cite{Al91} showed that the aforementioned conjecture of Granville implies the following weaker estimate on $\SF([n])$, which will serve as a second example application of Lemma~\ref{lemma:KW-basic}.

\begin{thm}[{\cite{Al91}}]
  \label{thm:CE-weak}
  The set $\{1, \ldots, n\}$ has at most $2^{(1/2+o(1))n}$ sum-free subsets.
\end{thm}

The Cameron--Erd\H{o}s conjecture was solved some fifteen years later by Green~\cite{Gr04} and, independently, by Sapozhenko~\cite{Sa03}. The solution due to Sapozhenko uses a method akin to the Kleitman--Winston algorithm presented in Section~\ref{sec:KW}, while the one due to Green uses discrete Fourier analysis.\footnote{However, one might still argue that the general `philosophy' behind Green's proof is similar.} We do not discuss either of their arguments here, but instead refer the interested reader to the original papers. Finally, we mention that strong estimates on the number of sum-free subsets of $[n]$ with a given number of elements, which imply the conjecture, were recently obtained in~\cite{AlBaMoSa14-CE}; the~proof there employs the ideas presented in Section~\ref{sec:KW}.

\begin{proof}[Proof of Theorem~\ref{thm:CE-weak}]
  Observe first that the number of all subsets of $[n]$ which contain fewer than $n^{2/3}$ elements from $\{1, \ldots, \lceil n/2 \rceil - 1\}$ is at most $(n/2)^{n^{2/3}} 2^{n/2+1}$. Therefore, we may restrict our attention to sum-free sets that contain at least $n^{2/3}$ elements strictly smaller than $n/2$. For each such set $A$, let $S_A$ be the set of $\lfloor n^{2/3} \rfloor$ smallest elements of $A$.

  Given a set $S \subseteq \{1, \ldots, \lceil n/2 \rceil - 1\}$, define an auxiliary graph $G_S$ with vertex set $[n]$ by letting
  \[
  E(G_S) = \{xy \colon \text{$x + s \equiv y \pmod n$ for some $s \in S \cup (-S)$}\}
  \]
  and note that $G_S$ is $2|S|$-regular, as $n - (\lceil n/2 \rceil - 1) > \lceil n/2 \rceil -1$ and hence $S$ and $-S$ contain different residues modulo $n$. The~crucial observation is that for every sum-free $A$ as above, the set $A \setminus S_A$ is an independent set in the graph $G_{S_A}$. Indeed, otherwise there would be $x, y \in A \setminus S_A$ and an $s \in S_A \cup (-S_A)$  with $x + s \equiv y \pmod n$; since $1 \le |s| < x, y \le n$, this is only possible when $x + s = y$. In particular, for a given $S \subseteq \{1, \ldots, \lceil n / 2 \rceil -1\}$, there are at most $i(G_S)$ sum-free sets~$A$ satisfying $S = S_A$. By Theorem~\ref{thm:Sapozhenko}, we conclude that
  \[
  \SF([n]) \le (n/2)^{n^{2/3}} 2^{n/2+1} + \binom{n/2}{n^{2/3}} \cdot 2^{\left(1 + O(n^{-1/3}\sqrt{\log n})\right)\frac{n}{2}} \le 2^{\left(1/2 + O(n^{-1/3} \log n)\right)n}.\qedhere
  \]
\end{proof}

Before closing this section, we remark that the paper of Alon~\cite{Al91} started a very successful line of inquiry into the closely related problem of determining the number of sum-free sets contained in an arbitrary finite abelian group; see, e.g., \cite{AlBaMoSa14-AG, GrRu04, GrRu05, LeLuSc01, Sa02}. In many of these works, variations of the ideas presented in Section~\ref{sec:KW} play a prominent role.

\subsection{Independent sets in regular graphs without small eigenvalues}

Since every $n$-vertex bipartite graph $G$ satisfies $\alpha(G) \ge n/2$ and hence it contains at least $2^{n/2}$ independent sets, the upper bounds for $i(G)$ proved in Section~\ref{sec:indep-sets-reg-graphs} are essentially best possible whenever $G$ is bipartite. It is natural to ask whether these bounds can be improved when one assumes that $G$ is `far' from being bipartite. An affirmative answer to this question was given by Alon and R\"odl~\cite{AlRo05}.

Recall that the adjacency matrix of an $n$-vertex graph $G$ is a real-valued symmetric $n \times n$ matrix and therefore it has $n$ real eigenvalues. Denote these eigenvalues by $\lambda_1, \ldots, \lambda_n$, where $\lambda_1 \ge \ldots \ge \lambda_n$. It is well known that the quantity $\max\{|\lambda_2|, |\lambda_n|\}$, called the \emph{second eigenvalue} of $G$, is closely tied with, among other parameters, the expansion properties of $G$. We shall be interested only in the smallest eigenvalue $\lambda_n$ of $G$, which we denote by $\lambda(G)$. It was first proved by Hoffman~\cite{Ho70} that every $d$-regular $n$-vertex graph $G$ satisfies $\alpha(G) \le \frac{-\lambda(G)}{d-\lambda(G)}n$. This was later significantly strengthened\footnote{In particular, Lemma~\ref{lemma:Alon-Chung} implies that $e_G(A) > 0$ for every $A$ with more than $\frac{-\lambda(G)}{d-\lambda(G)}n$ vertices.} by Alon and Chung~\cite{AlCh88}, who established the following relation between~$\lambda(G)$ and the number of edges induced by large sets of vertices in $G$, cf.\ the expander mixing lemma (see, e.g.,~\cite{HoLiWi06}).

\begin{lemma}[{\cite{AlCh88}}]
  \label{lemma:Alon-Chung}
  Let $G$ be an $n$-vertex $d$-regular graph. For all $A \subseteq V(G)$,
  \[
  2e_G(A) \ge \frac{d}{n}|A|^2 + \frac{\lambda(G)}{n}|A|\big(n-|A|\big).
  \]
\end{lemma}

Alon and R\"odl~\cite{AlRo05} were the first to prove that if $\lambda(G)$ is much larger than $-d$, then each such $G$ has far fewer than $2^{n/2}$ independent sets. As our next application of Lemma~\ref{lemma:KW-basic}, we derive a similar estimate, originally proved in~\cite{AlBaMoSa14-AG}.

\begin{thm}[{\cite{AlBaMoSa14-AG}}]
  \label{thm:eigenvalue}
  For every $\eps > 0$, there exists a constant $C$ such that the following holds. If $G$ is an~$n$-vertex $d$-regular graph with $\lambda(G) \ge -\lambda$, then
  \[
  i(G,m) \le \binom{\left( \frac{\lambda}{d+\lambda} + \eps \right) n}{m},
  \]
  provided that $m \ge Cn/d$.
\end{thm}

\begin{proof}[Proof of Theorem~\ref{thm:eigenvalue}]
  Fix some $\eps > 0$, let $G$ be an $n$-vertex $d$-regular graph, and let $\lambda = -\lambda(G)$. We may assume that $\frac{\lambda}{d+\lambda} + \eps < 1$ as otherwise there is nothing to prove. Let $U \subseteq V(G)$ be an arbitrary set with $|U| \ge \left(\frac{\lambda}{d+\lambda} + \frac{\eps}{2}\right)n$. Lemma~\ref{lemma:Alon-Chung} implies that
  \[
  2e_G(U) \ge \frac{d}{n} |U|^2 - \frac{\lambda}{n}|U|\big(n-|U|\big) = \frac{|U|}{n} \big((d+\lambda)|U| - \lambda n \big) \ge \frac{\eps d}{2}|U| \ge \frac{\eps d}{n} \binom{|U|}{2}.
  \]
  Let $\beta = \frac{\eps d}{n}$, $q = \left\lceil \frac{\log (2/\eps)}{\eps} \cdot \frac{n}{d}\right\rceil$, and $R = \left(\frac{\lambda}{d+\lambda} + \frac{\eps}{2}\right)n$ and observe that $R \ge e^{-\beta q}n$. If follows from Lemma~\ref{lemma:KW-basic} that for every $m$ with $m \ge q$,
  \begin{equation}
    \label{eq:eigenvalue-iGm-bound}
    i(G,m) \le \binom{n}{q} \binom{R}{m-q}.
  \end{equation}
  Let $r(t)$ denote the right-hand side of~\eqref{eq:eigenvalue-iGm-bound} with $q$ replaced by $t$. We may clearly assume that $m \le \alpha(G) \le \frac{\lambda}{d+\lambda}n$, as otherwise $i(G,m) = 0$. An elementary calculation shows that
  \[
  \frac{r(t+1)}{r(t)} = \frac{n-t}{t+1} \cdot \frac{m-t}{R - m + t + 1} \le \frac{nm}{(t+1)(R-m)} \le \frac{2m}{\eps(t+1)}
  \]
  and hence
  \[
  i(G,m) = r(q) = \prod_{t=0}^{q-1} \frac{r(t+1)}{r(t)} \cdot r(0) \le \frac{(2m)^q}{\eps^q q!} \cdot \binom{R}{m} \le \left(\frac{2em}{\eps q}\right)^q \cdot \left(\frac{R}{R + \eps n /2}\right)^m \binom{R + \eps n/ 2}{m},
  \]
  where we used the inequalities $a! > (a/e)^a$ and $\binom{a}{c} \ge (a/b)^c \binom{b}{c}$ valid whenever $a \ge b \ge c \ge 0$. Finally, if $K$ is sufficiently large (as a function of $\eps$) and $C \ge K \cdot \left\lceil \frac{\log(2/\eps)}{\eps} \right\rceil$, then for every $m$ with $m \ge Cn/d \ge Kq$,
  \[
  \left( \frac{2em}{\eps q} \right)^{q/m} \cdot \frac{R}{R + \eps n /2} \le \left(\frac{2Ke}{\eps}\right)^{1/K} \cdot \left(1-\frac{\eps}{2}\right) \le 1,
  \]
  which completes the proof of the theorem.
\end{proof}

We close this section with several remarks. First, the constant $\frac{\lambda}{d+\lambda}$ in the assertion of the theorem is optimal as for many values of $n$, $d$, and $\alpha$, there are $n$-vertex $d$-regular graphs with $\alpha(G) = \frac{-\lambda(G)}{d-\lambda(G)} n = \alpha n$. Second, the assumption that $m \ge Cn/d$ cannot be relaxed as for every~$\eps > 0$, every $n$-vertex $d$-regular graph $G$ satisfies $i(G,m) \ge \binom{(1-\eps)n}{m}$ whenever $m \le \eps n/(d+1)$. (To~see this, consider the greedy process of constructing an independent set which repeatedly picks an~arbitrary vertex of $G$ and removes it and all of its neighbours from $G$.) Third, the above theorem implies the conjecture of Granville stated in Section~\ref{sec:indep-sets-reg-graphs} as $\lambda(G) \ge -d$ for every $d$-regular graph $G$. Finally, we refer the interested reader to~\cite{AlBaMoSa14-AG} and~\cite{AlRo05}, where Theorem~\ref{thm:eigenvalue} was used to obtain upper bounds on the number of sum-free sets in abelian groups of even order and lower bounds on some multicolor Ramsey numbers, respectively.

\subsection{The number of $C_4$-free graphs}

\label{sec:number-C4-free-graphs}

As our next example, we present the main result from one of the papers of Kleitman and Winston~\cite{KlWi82} which introduced the methods described in Section~\ref{sec:KW}. Call a graph \emph{$C_4$-free} if it does not contain a cycle of length four and let $\ex(n,C_4)$ denote the maximum number of edges in a $C_4$-free graph with $n$ vertices. A classical result of K\H{o}v\'ari, S\'os, and Tur\'an~\cite{KoSoTu54} together with a construction due to Brown~\cite{Br66} and Erd\H{o}s, R\'enyi, and S\'os~\cite{ErReSo66} imply that
\[
\ex(n,C_4) = \left(\frac{1}{2} + o(1)\right) n^{3/2}.
\]
Let $f_n(C_4)$ be the number of (labeled) $C_4$-free graphs on the vertex set $\{1, \ldots, n\}$. Since each subgraph of a $C_4$-free graph is itself $C_4$-free, we have
\[
2^{\ex(n,C_4)} \le f_n(C_4) \le \sum_{m=0}^{\ex(n,C_4)} \binom{\binom{n}{2}}{m} = 2^{\Theta(\ex(n,C_4) \log n)},
\]
which yields
\begin{equation}
  \label{eq:fnC4-trivial}
  \ex(n,C_4) \le \log_2 f_n(C_4) \le O\big(\ex(n,C_4) \log n\big).
\end{equation}
Answering a question of Erd\H{o}s, Kleitman and Winston~\cite{KlWi82} showed that the lower bound in~\eqref{eq:fnC4-trivial} is tight up to a constant factor.

\begin{thm}[{\cite{KlWi82}}]
  \label{thm:KlWi-C4}
  There is a positive constant $C$ such that
  \[
  \log_2 f_n(C_4) \le Cn^{3/2}.
  \]
\end{thm}

Before we continue with the proof of the theorem, let us make a few comments. In fact, Erd\H{o}s asked whether $\log_2 f_n(H) = (1+o(1))\ex(n,H)$ for an arbitrary $H$ that contains a cycle. This was shown to be the case by Erd\H{o}s, Frankl, and R\"odl~\cite{ErFrRo86} under the assumption that $\chi(H) \ge 3$. Very recently, Morris and Saxton~\cite{MoSa14} proved that $\log_2 f_n(C_6) \ge 1.0007 \cdot \ex(n,C_6)$ for infinitely many~$n$. But the notoriously difficult problem of determining whether or not $\log_2 f_n(H) = O(\ex(n,H))$ for every bipartite $H$ that is not a forest remains unsolved, apart from the following two special cases: $H$ is a cycle length four~\cite{KlWi82}, six~\cite{KlWi96}, or ten~\cite{MoSa14} or $H$ is an unbalanced complete bipartite graph~\cite{BaSa11-Kmm,BaSa11-Kst}. More exactly, it is proved in~\cite{BaSa11-Kst} and~\cite{MoSa14} that $\log_2 f_n(K_{s,t}) = O(n^{2-1/s})$ whenever $2 \le s \le t$ and that $\log_2 f_n(C_{2\ell}) = O(n^{1+1/\ell})$ for every $\ell \ge 2$, respectively. As it is commonly believed that $\ex(n, K_{s,t}) = \Omega(n^{2-1/s})$ whenever $s \le t$ and that $\ex(n,C_{2\ell}) = \Omega(n^{1+1/\ell})$, both these results are most likely best possible. Finally, we mention that the proofs of most of the results mentioned in this paragraph use either a variant of Lemma~\ref{lemma:KW-basic} or extensions of the ideas presented in Section~\ref{sec:KW} to hypergraphs, see Section~\ref{sec:extensions-to-hypergraphs}.

\begin{proof}[{Proof of Theorem~\ref{thm:KlWi-C4}}]
  Note that one can order the vertices of every $n$-vertex graph $G$ as $v_1, \ldots, v_n$ in such a way that for every $i \in \{2, \ldots, n\}$, letting $G_i = G[\{v_1, \ldots, v_i\}]$,
  \[
  \delta(G_{i-1}) \ge \deg_{G_i}(v_i) - 1.
  \]
  Indeed, one may obtain such an ordering by iteratively letting $v_i$ be a minimum-degree vertex of $G - \{v_{i+1}, \ldots, v_n\}$ for $i = n, \ldots, 2$. In particular, every labeled $n$-vertex graph $G$ can be constructed in the following way. First, choose an ordering $v_1, \ldots, v_n$ of the vertices and let $G_1$ be the empty graph with vertex set~$\{v_1\}$. Second, for each $i \in \{2, \ldots n\}$, build a graph $G_i$ by adding to the graph $G_{i-1}$ a vertex labeled $v_i$ in such a way that its degree $d_i$ (in $G_i$) satisfies $d_i \le \delta(G_{i-1})+1$. Finally, we let $G = G_n$. Observe that $G$ is $C_4$-free if and only if $G_i$ is $C_4$-free for each $i$.

  Now, given integers~$d$ and $i$ with $d \le i$, let $g_i(d)$ denote the maximum number of ways to attach a vertex of degree $d$ to an $i$-vertex $C_4$-free graph with minimum degree at least $d-1$ in such a way that the resulting graph remains $C_4$-free. This number is well defined as clearly $g_i(d) \le \binom{i}{d}$. Moreover, let $g_i = \max \{g_i(d) \colon d \le i\}$. The argument given in the previous paragraph proves that
  \begin{equation}
    \label{eq:fnC4-crude}
    f_n(C_4) \le n! \cdot n! \cdot \prod_{i=2}^n g_{i-1}.
  \end{equation}
  Indeed, there are $n!$ ways to order $[n]$ as $v_1, \ldots, v_n$ and for each such ordering, there are at most $n!$ choices for the sequence $d_2, \ldots, d_n$ of degrees. In view of~\eqref{eq:fnC4-crude}, the following claim easily implies the assertion of the theorem.

  \begin{claim}
    There exists a constant $C$ such that $g_n \le \exp(C\sqrt{n})$ for all $n$.
  \end{claim}
  
  Without loss of generality, we may assume that $n$ is large. Thus, if $d \le \sqrt{n} / \log n$, then
  \[
  g_n(d) \le \binom{n}{d} \le \binom{n}{\frac{\sqrt{n}}{\log n}} \le \left(e\sqrt{n}\log n\right)^{\frac{\sqrt{n}}{\log n}} \le \exp(\sqrt{n}).
  \]
  Therefore, we shall from now on assume that $d > \sqrt{n} / \log n$. Let $G$ be a $C_4$-free graph on $n$ vertices with $\delta(G) \ge d-1$. Let $H$ be the square of $G$, that is, the graph with $V(H) = V(G)$ and
  \[
  E(H) = \{xy \colon xz, yz \in E(G) \text{ for some $z \in V(G)$}\}.
  \]
  Crucially, observe that adding $v$ to $G$ will result in a $C_4$-free graph if and only if the neighbourhood of $v$ is an independent set in $H$. Hence, $i(H,d)$ is an upper bound on the number of $C_4$-free extensions of $G$ by a vertex of degree $d$. We shall estimate $i(H, d)$ using Lemma~\ref{lemma:KW-basic}.

  To this end, we show that subgraphs of $H$ induced by large subsets of $V(H)$ have reasonably high density. Since $G$ is $C_4$-free, every edge $xy$ of $H$ corresponds to a unique vertex $z \in V(G)$ such that $xz$ and $yz$ are edges of $G$. Therefore, for each $B \subseteq V(H)$,
  \[
  e_H(B) = \sum_{z \in V(G)} \binom{\deg_G(z,B)}{2} \ge n \cdot \binom{\sum_z \deg(z,B) / n}{2},
  \]
  where the last inequality is Jensen's inequality applied to the convex function $x \mapsto \binom{x}{2}$. Since
  \[
  \sum_{z \in V(G)}\deg_G(z,B) = \sum_{x \in B}\deg_G(x) \ge |B| \cdot \delta(G) \ge (d-1)|B|,
  \]
  then assuming that $|B| \ge \frac{2n}{d-1}$ implies
  \[
  e_H(B) \ge n \cdot \frac{(d-1)|B|}{2n} \left(\frac{(d-1)|B|}{n} - 1\right) \ge \frac{(d-1)^2}{2n}\binom{|B|}{2}.
  \]

  Finally, let $R = \frac{2n}{d-1}$, $\beta = \frac{(d-1)^2}{2n}$, and $q = \lceil 3(\log n)^3\rceil$. Since $d > \sqrt{n}/\log n$ and $n$ is large, then $\beta q \ge \log n$ and therefore $e^{-\beta q} n \le 1 \le R$. If follows from Lemma~\ref{lemma:KW-basic} that
  \[
  i(H, d) \le \binom{n}{q} \binom{\frac{2n}{d-1}}{d-q} \le e^{4\log^4n} \cdot \left(\frac{2en}{(d-q)^2}\right)^{d-q} \le \sup_{k > 0}  \left(\frac{e\sqrt{n}}{k}\right)^{2k} = e^{2\sqrt{n}},
  \]
  where we used the assumption that $n$ is large and the fact that $\sup\left\{\left(\frac{e}{x}\right)^x \colon x > 0\right\} = e$.
\end{proof}

\subsection{Roth's theorem in random sets}

As our final example, we present a short proof of a well-known result of Kohayakawa, \L uczak, and R\"odl~\cite{KoLuRo96}. Recall that $[n]$ denotes the set $\{1, \ldots, n\}$. A famous theorem of Roth~\cite{Ro53} asserts that for every positive $\delta$, any set of at least $\delta n$ integers from~$[n]$ contains a $3$-term arithmetic progression ($3$-term AP), provided that $n$ is sufficiently large (as a~function of $\delta$ only). Given a positive $\delta$, we shall say that a set $A \subseteq \ZZ$ is \emph{$\delta$-Roth} if each $B \subseteq A$ satisfying $|B| \ge \delta |A|$ contains a $3$-term AP. We may now restate Roth's theorem as follows. For every positive $\delta$, there exists an $n_0$ such that the set $[n]$ is $\delta$-Roth whenever $n \ge n_0$. With the aim of showing that there exist `smaller' and `sparser' $\delta$-Roth sets Kohayakawa, \L uczak, and R\"odl~\cite{KoLuRo96} proved the following result.

\begin{thm}[{\cite{KoLuRo96}}]
  \label{thm:KoLuRo}
  For every positive $\delta$, there exists a constant $C$ such that if $C\sqrt{n} \le m \le n$, then the probability that a uniformly chosen random $m$-element subset of $\{1, \ldots, n\}$ is $\delta$-Roth tends to~$1$ as $n \to \infty$.
\end{thm}

We shall deduce Theorem~\ref{thm:KoLuRo} as an easy corollary of the following upper bound for the number of subsets of $[n]$ of a given cardinality that do not contain a $3$-term AP, originally proved in~\cite{BaMoSa14} and~\cite{SaTh14} in a much more general form. This upper bound will be derived from Roth's theorem using Lemma~\ref{lemma:KW-containers} with one additional twist which was previously considered in~\cite{AlBaMoSa14-AG}.

\begin{thm}
  \label{thm:3AP-free-count}
  For every positive $\eps$, there exists a constant $D$ such that if $D\sqrt{n} \le m \le n$,
  \[
  \left|\big\{A \subseteq [n] \colon \text{$|A| = m$ and $A$ contains no $3$-term AP}\big\}\right| \le \binom{\eps n}{m}.
  \]
\end{thm}

\begin{proof}[Proof of Theorem~\ref{thm:KoLuRo}]
  Fix a positive $\delta$, let $\eps = \delta / 6$, and let $D$ be the constant from the statement of Theorem~\ref{thm:3AP-free-count}. Let $C = D / \delta$ and suppose that $C \sqrt{n} \le m \le n$. Since $\lceil \delta m \rceil \ge D \sqrt{n}$, Theorem~\ref{thm:3AP-free-count} implies that the set $\cA$ defined by
  \[
  \cA = \big\{ A \subseteq [n] \colon \text{$|A| = \lceil \delta m \rceil$ and $A$ contains no $3$-term AP}\big\}
  \]
  has at most $\binom{\eps n}{\lceil \delta m \rceil}$ elements. Now, let $R$ be an $m$-element subset of $[n]$ chosen uniformly at random. Clearly,
  \[
  \begin{split}
    \Pr\big(\text{$R$ is not $\delta$-Roth}\big) & = \Pr\big(\text{$R \supseteq A$ for some $A \in \cA$}\big) \le \sum_{A \in \cA} \Pr(R \supseteq A) \le \sum_{A \in \cA} \left(\frac{m}{n}\right)^{|A|} \\
    & = |\cA| \cdot \left(\frac{m}{n}\right)^{\lceil \delta m\rceil} \le \binom{\eps n}{\lceil \delta m \rceil} \cdot \left(\frac{m}{n}\right)^{\lceil \delta m\rceil} \le \left(\frac{\eps e n}{\lceil \delta m \rceil} \cdot \frac{m}{n} \right)^{\lceil \delta m \rceil} \le 2^{-\delta m}.\qedhere
  \end{split}
  \]
\end{proof}

Our proof of Theorem~\ref{thm:3AP-free-count} will use the following simple consequence of Roth's theorem, observed first by Varnavides~\cite{Va59}, as a `black box'.

\begin{prop}[{\cite{Ro53,Va59}}]
  \label{prop:Varnavides}
  For every positive $\delta$, there exist an integer $n_0$ and a positive $\beta$ such that if $n \ge n_0$, then every set of at least $\delta n$ integers from $\{1, \ldots, n\}$ contains at least $\beta n^2$ $3$-term APs.
\end{prop}

\begin{proof}[Proof of Theorem~\ref{thm:3AP-free-count}]
  Fix a positive $\eps$, let $n_0$ and $\beta$ be the constants from the statement of Proposition~\ref{prop:Varnavides} invoked with $\delta = \eps / 2$, and suppose that $n \ge n_0$. Given an arbitrary set $B \subseteq [n]$ and integers~$m$ and $\np$, let
  \begin{align*}
    a(B,m) & = \left|\big\{ I \subseteq B \colon \text{$|I| = m$ and $I$ contains no $3$-term AP}\big\}\right|, \\
    a(\np,m) & = \max \big\{ a(B,m) \colon \text{$B \subseteq [n]$ with $|B| = \np$} \big\}.
  \end{align*}
  Our aim is to show that $ a([n],m) = a(n,m)\le \binom{\eps n}{m}$, provided that $m \ge C \sqrt{n}$ for some constant $C$ which depends only on $\eps$. This inequality will follow from the trivial observation that $a(\np,m) \le \binom{\np}{m}$ for all $\np$ and $m$ and the following claim.

  \begin{claim}
    If $\np \ge \eps n / 2$ and $m \ge 2\lfloor\sqrt{n}\rfloor$, then $a(\np,m) \le 2 \binom{n}{\lfloor\sqrt{n}\rfloor}^2 \cdot a\big(\np - \lceil \beta n/12 \rceil, m-2\lfloor\sqrt{n}\rfloor\big)$.
  \end{claim}

  Let $\HH$ be the $3$-uniform hypergraph with vertex set $[n]$ whose edges are all triples of numbers which form a $3$-term AP. Let $B$ be an arbitrary $\np$-element subset of $[n]$. By Proposition~\ref{prop:Varnavides}, $e_{\HH}(B) \ge \beta n^2$. Let $Z \subseteq B$ be the set of all vertices of $\HH[B]$, the subhypergraph of $\HH$ induced by~$B$, whose degree is at least $\beta n$. In other words, $Z$ is the set of all numbers in $B$ that belong to at least $\beta n$ three-term APs contained in $B$. Since the maximum degree of $\HH$ is at most $2n$, we have $|Z| \ge \beta n$.

  We first estimate the number of $m$-element subsets of $B$ with no $3$-term AP that contain fewer than $\sqrt{n}$ elements of $Z$. Since each such set $A$ may be partitioned into $A_1$ and $A_2$, where $|A_1| = \lfloor \sqrt{n} \rfloor$ and $A_2 \subseteq B \setminus Z$, there are at most $\binom{n}{\lfloor \sqrt{n} \rfloor} \cdot a(\np - \lceil \beta n \rceil, m - \lfloor \sqrt{n} \rfloor)$ such sets. We may therefore focus on counting subsets of $B$ that contain at least $\sqrt{n}$ elements of $Z$. We shall obtain a suitable upper bound for their number using Lemma~\ref{lemma:KW-containers}.

  Let $W$ be an arbitrary subset of $Z$ and consider the auxiliary graph $G_W$ with vertex set $B$ whose edges are all pairs $\{x,y\}$ such that $\{x,y,z\} \in \HH$ for some $z \in W$. Since for a given pair $\{x, y\} \subseteq [n]$, there are at most three different $z$ such that $\{x,y,z\} \in \HH$, it follows that $e(G_W) \ge |W| \beta n/3$ and the maximum degree of $G_W$ is no more than $3|W|$. It follows that for an arbitrary subset $U$ of $B$ with at least $\np - \beta n /12$ elements,
  \begin{equation}
    \label{eq:eGW-U}
    e_{G_W}(U) \ge e(G_W) - |B \setminus U| \cdot \Delta(G_W) \ge \frac{\beta n |W|}{3} - \frac{\beta n}{12} \cdot 3|W| = \frac{\beta n |W|}{12}.
  \end{equation}
  Observe crucially that if some set $I \cup W$ contains no $3$-term APs, then $I$ is an independent set in the graph $G_W$.

  Let $w = \lfloor \sqrt{n} \rfloor$ and fix some $W \subseteq Z$ with $|W| = w$. We shall prove an upper bound on the~number of ways one can extend $W$ to an $m$-element subset of $B$ that contains no $3$-term APs. By our above discussion, if $I \cup W$ is such a set, then $I$ is an independent set of $G_W$ with $m-w$ elements. Let $\cS$ be the family of sets and let $f$ and $g$ be the maps whose existence is postulated by Lemma~\ref{lemma:KW-containers} with $G = G_W$, $q = \lfloor \sqrt{n} \rfloor$, $R = n' - \lceil \beta n /12 \rceil$, and $D = \beta w / 6$. Note that the assumptions of the lemma are satisfied by our discussion above, see~\eqref{eq:eGW-U}.
  Since clearly for each extension $I$ of $W$ to an $m$-element subset of $B$ with no $3$-term APs, $I \cap f(g(I))$ contains no $3$-term APs, the~number $E_W$ of extensions of $W$ satisfies
  \[
  E_W \le \sum_{S \in \cS} a\big(f(S), m-w-q\big) \le \binom{n}{q} \cdot a\big(R, m-w-q\big).
  \]

  We conclude that
  \[
  \begin{split}
    a(B,m) & \le \binom{n}{\lfloor\sqrt{n}\rfloor} \cdot a\big(\np - \lceil \beta n \rceil, m - \lfloor\sqrt{n}\rfloor\big) + \sum_{W \subseteq Z \colon |W| = w} E_W \\
    & \le \binom{n}{\lfloor\sqrt{n}\rfloor}^2 \cdot a\big(\np - \lceil \beta n \rceil, m - 2\lfloor\sqrt{n}\rfloor\big) + \binom{n}{w}\binom{n}{q} \cdot a\big(\np-\lceil \beta n /12 \rceil, m-2\lfloor\sqrt{n}\rfloor\big) \\
    & \le 2 \binom{n}{\lfloor\sqrt{n}\rfloor}^2 \cdot a\big(\np - \lceil \beta n /12 \rceil, m - 2\lfloor \sqrt{n} \rfloor\big),
  \end{split}
  \]
  which, since $B$ was an arbitrary $\np$-element subset of $[n]$, proves the claim.

  \medskip

  Let $K = \lceil (12-6\eps)/\beta \rceil$ and suppose that $m \ge \sqrt{n}$. We recursively invoke the claim $K$ times to obtain
  \begin{equation}
    \label{eq:anm-unprocessed}
    a(n,m) \le 2^K \binom{n}{\lfloor \sqrt{n} \rfloor}^{2K} \binom{\eps n/2}{m-2K\lfloor \sqrt{n} \rfloor} \le 2^K \binom{2Kn}{2K\lfloor \sqrt{n} \rfloor} \binom{\eps n/2}{m-2K\lfloor \sqrt{n} \rfloor}.
  \end{equation}
  As in the proof of Theorem~\ref{thm:eigenvalue}, denote by $r(t)$ the right-hand side of~\eqref{eq:anm-unprocessed} with $2K \lfloor \sqrt{n} \rfloor$ replaced by~$t$. We may clearly assume that $m < \eps n/4$ as otherwise $a(n,m) = 0$ by Roth's theorem (we may assume that $n$ is sufficiently large). An~elementary calculation shows that
  \[
  \frac{r(t+1)}{r(t)} = \frac{2Kn-t}{t+1} \cdot \frac{m - t}{\eps n /2 -m + t + 1} \le \frac{2Knm}{(t+1)(\eps n / 2 - m)} \le \frac{8Km}{\eps(t+1)}
  \]
  and hence, letting $T = 2K \lfloor \sqrt{n} \rfloor$,
  \[
  a(n,m) \le r(T) \le 2^K \cdot \frac{(8Km)^T}{\eps^T T!} \cdot \binom{\eps n/2}{m} \le 2^K \cdot \left(\frac{8eKm}{\eps T}\right)^T \cdot \left(\frac{1}{2}\right)^m \binom{\eps n}{m}.
  \]
  Finally, if $D$ is sufficiently large as a function of $K$ and $\eps$, then for every $m$ with $m \ge D\sqrt{n} \ge D/(2K) \cdot T$, we have
  \[
  2^{K/m} \cdot \left(\frac{8eKm}{\eps T}\right)^{T/m} \le 2,
  \]
  which completes the proof of the theorem.
\end{proof}

\section{Concluding remarks and further reading}

\subsection{Other applications of the Kleitman--Winston method}

There have been quite a few successful applications of the Kleitman--Winston method other than the ones presented in Section~\ref{sec:applications}. In particular, variants of Lemma~\ref{lemma:KW-basic} were used in the following works: Kleitman and Wilson~\cite{KlWi96} proved that the number of $n$-vertex graphs with girth larger than $2\ell$ is $2^{O(n^{1+1/\ell})}$; Dellamonica, Kohayakawa, Lee, R\"odl, and the author~\cite{DeKoLeRoSa14-B3, DeKoLeRoSa14-Bh, KoLeRoSa14} proved sharp bounds on the number of subsets of $[n]$ with a given cardinality which contain no non-trivial solutions to the equation $a_1 + \ldots + a_h = b_1 + \ldots + b_h$ for every $h \ge 2$; Balogh, Das, Delcourt, Liu, and Sharifzadeh~\cite{BaDaDeLiSh14} and Gauy, H\`an, and Oliveira~\cite{GaHaOl14} proved sharp bounds for the number of intersecting families of $k$-element subsets of $[n]$ with a given cardinality and for the typical size of the largest intersecting subfamily contained in a random collection of $k$-element subsets of $[n]$.

\subsection{Extensions of the Kleitman--Winston method to hypergraphs}

\label{sec:extensions-to-hypergraphs}

It seems natural to seek a generalisation of the Kleitman--Winston method that would yield non-trivial upper bounds for the number of independent sets in a hypergraphs of higher uniformity. Perhaps somewhat surprisingly, such generalisations were considered only fairly recently. To the best of our knowledge this was first done in~\cite{BaSa11-Kmm,BaSa11-Kst}, where sharp upper bounds for the number of $n$-vertex graphs which do not contain a copy of a fixed complete bipartite subgraph were proved using a generalisation of the argument presented in Section~\ref{sec:number-C4-free-graphs}. Around the same time, similar ideas were developed by Saxton and Thomason, who used them to establish lower bounds for the list chromatic number of regular uniform hypergraphs~\cite{SaTh12}. Inspired by the groundbreaking work of Conlon and Gowers~\cite{CoGo14} and Schacht~\cite{Sc14}, these efforts culminated in far-reaching generalisations of the Kleitman--Winston method to arbitrary uniform hypergraphs, obtained independently by Saxton and Thomason~\cite{SaTh14}, and by Balogh, Morris, and the author~\cite{BaMoSa14}. For further details, we refer the interested reader to~\cite{BaMoSa14, Co14, CoGo14, Sa14, SaTh14, Sc14}.

\bigskip
\noindent
\textbf{Acknowledgments.} I would like to thank Noga Alon, J\'ozsi Balogh, Domingos Dellamonica, Yoshi Kohayakawa, Sang June Lee, Rob Morris, and Vojta R\"odl for many interesting discussions on the topics of independent sets in graphs and the Kleitman--Winston method and its applications over the past several years. These discussions have greatly influenced the content of this paper. I would also like to thank David Conlon, Asaf Ferber, and Rob Morris for their careful reading of an earlier version of this manuscript and many valuable comments which helped me improve the exposition and saved me from making several embarrassing mistakes. Finally, special thanks to Jarik Ne\v{s}et\v{r}il for his encouragement to write this survey.

\bibliographystyle{amsplain}
\bibliography{KW-survey}

\end{document}